\documentclass[a4]{amsart}
\usepackage{amsmath}
\usepackage{amssymb}
\usepackage{mathrsfs}
\usepackage{amsfonts} 
\usepackage{amsthm}
\usepackage{enumerate}
\usepackage{hyperref}

\newtheorem{thm}{\textbf{Theorem}}[section]
\newtheorem{Def}[thm]{\textbf{Definition}}
\newtheorem{prop}[thm]{\textbf{Proposition}}
\newtheorem{lem}[thm]{\textbf{Lemma}}

\newtheorem{rem}[thm]{Remark}

\def\Q{\mathbb{Q}}
\def\Z{\mathbb{Z}}
\def\C{\mathbb{C}}
\def\A{\mathbb{A}}

\def\hom{\operatorname{Hom}}

\def\Irr{\operatorname{Irr}}
\def\GL{\operatorname{GL}}
\def\tr{\operatorname{tr}}

\title[$p$-adic functoriality]{\textbf{\textsc{ \MakeLowercase{$p$}-adic functoriality for inner forms of unitary groups in three variables}}}
\author{\textsc{Judith Ludwig}}
\date{}

\begin{document}
\maketitle

\begin{abstract} We prove $p$-adic functoriality for inner forms of unitary groups in three variables by establishing the existence of morphisms between eigenvarieties that extend the classical Langlands functoriality.  
\end{abstract}

\section{Introduction}

Eigenvarieties are rigid analytic spaces that $p$-adically interpolate systems of Hecke eigenvalues attached to automorphic representations. For definite unitary groups they have been constructed by Chenevier \cite{Chenevier} using Buzzard's and Coleman's approach \cite{kevin} and by Emerton \cite{emerton} using completed cohomology. Given a construction of eigenvarieties for a class of groups a question that naturally arises is whether one can extend Langlands functoriality $p$-adically. The aim is to construct a morphism of eigenvarieties that on classical points agrees with the classical transfer. Chenevier was able to prove \textit{p}-adic functoriality in the case of the Jacquet--Langlands transfer by interpolating the classical transfer \cite{pjl}. Using his approach White \cite{pjw} established a $p$-adic endoscopic transfer for unitary groups, Flicker \cite{flicker} proved a $p$-adic interpolation of base-change and Newton \cite{newton} proved a $p$-adic Jacquet--Langlands correspondence in Emerton's setting. 

In these notes we prove $p$-adic functoriality for certain inner forms $G$ and $G'$ of a quasi-split unitary group $U_3^*$ in three variables.
There are various problems one encounters when trying to interpolate the classical transfer: 
\begin{itemize}
	\item The classical transfer can only be proved for $L$-packets (or $A$-packets) rather than for single automorphic representations and in general the multiplicities of elements in a packet may vary and can even be zero, yet eigenvarieties parameterize systems of eigenvalues of Hecke operators that show up in certain spaces of $p$-adic automorphic forms.   
	\item An eigenvariety comes with a tame level. If we want to construct a $p$-adic functoriality morphism from an eigenvariety for $G$ that has arbitrary tame level to an eigenvariety for $G'$, we need to find a suitable tame level for $G'$.
\end{itemize}
We deal with these problems as follows:

We study ``stable'' situations, i.e., situations where the multiplicities within a $L$-packet are constant. This allows us to transfer automorphic representations rather than packets. In order to find suitable tame levels of $G$ and $G'$ that are respected by the classical transfer, we use the theory of the Bernstein decomposition and properties of the local Jacquet--Langlands transfer. 

Our main theorem is Theorem \ref{main}, where we prove the existence and uniqueness of the $p$-adic functoriality morphism.  We remark that the source can be an eigenvariety for $G$ that has \textit{arbitrary} tame level at the bad primes, i.e., the primes, where $G$ and $G'$ differ. In this aspect, these notes differ from the papers on $p$-adic functoriality cited above, where the open compact subgroup at such primes is usually assumed to be maximal.

\section*{Acknowledgements} I thank Kevin Buzzard for suggesting that I think about this topic. I thank him and Paul-James White for many helpful conversations and their comments on an earlier draft of this paper. I am very grateful to Sug Woo Shin for helpful correspondence and in particular for explaining the arguments needed for the multiplicity statement in Section 2. I thank Christian Johansson for many helpful conversations. Moreover I would to like to thank the anonymous referee for many helpful comments. 

I thank Imperial College for providing a Roth Studentship and EPSRC for their financial support of my doctoral studies.

\vspace{0.3cm}
\textbf{Notation.} All division algebras occurring in this paper will be central simple algebras. 

\section{Classical functoriality}
Let $E/ \Q$ be an imaginary quadratic extension and $U_3^*$ the quasi-split unitary group in three variables attached to $E/ \Q$.
Let $G,G'$ be two definite unitary groups that are inner forms of $U_3^*$. Note that for all places $p$ that are non-split in $E$ the local groups $G_p$ and $G'_p$ are quasi-split and therefore isomorphic. Assume that $G$ comes from a division algebra $D$. We denote by $S$ the set of places $p$ where $G_p \ncong G'_p$ and assume furthermore that $G'_p \cong \GL_3/\Q_p$ for all places $p \in S $. Denote by \textit{JL} the local Jacquet--Langlands transfer (see Th\'eor\`eme principal in \cite{bdkv}, p.34).

We define $S_{G}$ to be the set of all primes $p$ that split in $E$ and such that $D_w$ is ramified for $w|p$ and similarly for $G'$.
We fix inner automorphisms $\psi_G: G_{\overline{\Q}}\rightarrow (U_3^*)_{\overline{\Q}}$ and $\psi_{G'}: G'_{\overline{\Q}}\rightarrow (U_3^*)_{\overline{\Q}}.$
Let $\Pi(G)$ (resp.\ $\Pi(G')$, $\Pi(U_3^*)$) denote the set of global $L$-packets of $G$ (resp. $G'$, $U_3^*$) as defined in \cite{Rog} Section 14.6. (cf.\ Sections 12.2 and 14.4 for the definition of local $L$-packets). We would like to stress that for the transfer to work Rogawski has to enlarge certain local $L$-packets to what he calls $A$-packets (see \cite{Rog} p.\ 199) and by definition the representations in a global $L$-packet vary locally within a local $L$-packet or an $A$-packet.

In Sections 14.4 and 14.6, Rogawski constructs transfer maps $\psi_G: \Pi(G)\rightarrow \Pi(U_3^*)$ and $\psi_{G'}:\Pi(G')\rightarrow \Pi(U_3^*)$. As $G$ comes from a division algebra all $L$-packets of $G$ are stable, i.e., non-endoscopic (cf.\ \cite{Rog} Theorem 14.6.3 and p.\ 201 for the definition of a stable $L$-packet), so using his notation we get $\Pi(G)=\Pi_s(G)$. Furthermore, Rogawski proves the following proposition:
\begin{prop}
The map $\psi_{G}$ defines a bijection between $\Pi_s(G)$ and 
$$Im(\psi_{G}):=\{ \Pi \in \Pi_s(U_3^*) : \dim(\Pi)=1 \text{ or } \Pi_v \in \Pi^2(U_{3,v}^*) \text{ for all } v \in S_{G}\cup \{\infty\} \}.$$
\end{prop}
\begin{proof} This is Proposition 14.6.2 in \cite{Rog}. 
\end{proof}
\begin{lem}
Let $H$ be an inner form of $U_3^*$. Assume $\Pi \in \Pi_s(H)$ is a stable $L$-packet. Then 
$$m(\pi)=1 \ \text{ for all } \pi \in \Pi.$$
\end{lem}
\begin{proof} By \cite{Rog} Equation (14.6.2) and the proof of Proposition 14.6.2, we get the following identity of distributions:
$$ \sum_{\pi \in \Pi}{m(\pi) \tr(\pi(f))}= \tr(\psi_H(\Pi)(f^*)),$$
where $f^*$ denotes the transfer of the function $f$ (see \cite{Rog} Section 14.2).
Using the fact that the transfer $\psi_H$ satisfies the character identity $\tr(\Pi(f))= \tr(\psi_H(\Pi)(f^*))$ (see Section 14.4) we get
$$ \sum_{\pi \in \Pi}{m(\pi) \tr(\pi(f))} = \tr(\Pi(f))= \sum_{\pi \in \Pi}{\left\langle 1,\pi \right\rangle \tr(\pi(f))}.$$
But $\left\langle 1,\pi \right\rangle \in \{\pm 1 \}$ and $m(\pi)\geq 0$, so we deduce that 
$$m(\pi)=1 \ \text{ for all } \pi \in \Pi .$$\end{proof}

\begin{thm}\label{transfer}
Let $\pi=\otimes \pi_p$ be an automorphic representation of $G$ which is not one-dimensional. Then 
$$\pi'\cong \bigotimes_{p\in S} JL(\pi_p) \otimes \bigotimes_{p\notin S} \pi_p$$
is an automorphic representation of $G'$ occurring with the same multiplicity as $\pi$. 
\end{thm}
\begin{proof}
Our assumptions on $G$ and $G'$ imply that $S_{G'}\subset S_{G}$ and so $Im(\psi_{G})\subset Im(\psi_{G'})$. Therefore, $\pi'$ is an element of a stable $L$-packet. We can therefore apply the previous lemma to conclude that $\pi'$ is automorphic.
\end{proof}

Following \cite{Rog} as in Section 14.6, we can also transfer one-dimensional representations: Let $\A_{E,1}$ denote the subgroup of norm one elements in $\A^*_E$, the ideles of $E$, and $E_1$ the norm one elements in $E^*$. 
Let $N_G:G(\A)\rightarrow \A_{E,1}$ be defined by $N_G(x)=\prod_{v} N_{G,v}(x_v)$, where $N_{G,v}$ is the determinant for $v\notin S_G$ and the reduced norm $\operatorname{Nrd}$ for $v\in S_G$. If $\pi$ is a one-dimensional automorphic representation of $G$, then $\pi=\chi \circ N_G$, where $\chi$ is a character of $\A_{E,1}$ trivial on $E_1$. The transfer of $\pi$ is the one-dimensional representation $\pi'=\chi \circ N_{G'}$ of $G'$ associated to this character. \\

Next we explore some properties of the local Jacquet--Langlands transfer that are needed to establish $p$-adic functoriality.
We make use of the Bernstein decomposition, the notation is as in \cite{bushnell}:
Let $F/\Q_p$ be a finite extension with ring of integers $\mathcal{O}_F$ and uniformizer $\varpi$, $A/F$ a central simple algebra and $A^*$ the group of units, so $A^* \cong \GL_m(D)$ for $D/F$ some division algebra. Let $d\in \mathbb{N}$ be such that $d^2=\operatorname{dim} D$. The reduced norm map will be denoted by $\operatorname{Nrd}:A^*\rightarrow F^*$. Let $O_D$ be a maximal order and $\varpi_D$ a uniformizer in $O_D$. 
The subgroups
$$ K_n:= \{x \in \GL_m(O_D) : x \equiv 1 (\text{ mod }\varpi_D^n) \}$$ 
form a fundamental system of compact open neighbourhoods of the identity.
Let $\mathcal{B}(A^*)$ be the set of Bernstein components of $A^*$. For $s \in \mathcal{B}(A^*)$ let $\mathcal{R}^s(A^*)$ be the category of isomorphism classes of smooth representations of $A^*$ associated to $s$, which we will also call a Bernstein component. Let
$ \Irr \mathcal{R}^s(A^*) \subset \mathcal{R}^s(A^*)$ be the class of objects that are irreducible representations. 

For $K$ a compact open subgroup of $A^*$ we write 
$$e_K:= \mu(K)^{-1}\mathbf{1}_K \in C^\infty_c(A^*,\C),$$
for the idempotent attached to $K$, here $\mathbf{1}_K$ is the characteristic function on $K$ and $\mu$ a Haar measure on $A^*$.

Now let $e \in C^\infty_c(A^*, \C)$ be an arbitrary idempotent. Then there exists a compact open subgroup $K\subset A^*$ such that $e_K\cdot e= e = e \cdot e_K$, where ``$\cdot$" denotes convolution. We choose $K$ to be of the form $K_n$.

We denote by $\Irr \mathcal{R}_e(A^*)$ the class of all irreducible smooth representations $\pi$ such that $e \cdot \pi\neq 0$, i.e., of those irreducible representations that are generated by $e \cdot \pi$. Note that for any such representation $\pi$ the set of $K$-fixed vectors is nonzero as
$$ 0 \neq e \cdot \pi = e \cdot e_K \cdot \pi= e \cdot \pi^K. $$   
It follows from \cite{bernstein} Corollaire 3.9
that there exists a finite set $\mathcal{J}_e\subset \mathcal{B}(A^*)$, such that 
$$ \Irr \mathcal{R}_e (A^*) \subset \bigcup_{s\in \mathcal{J}_e} \Irr \mathcal{R}^s(A^*).$$

\begin{prop}
Let $e \in C^\infty_c(A^*, \C)$ be an idempotent with the property that we can find a set $\mathcal{J}_e$ as above, such that any $s \in \mathcal{J}_e$ is associated to a supercuspidal representation. Then there exists a compact open subgroup $K_e$ of $\GL_{md}(F)$, such that
$$JL(\pi)^{K_e} \neq 0$$
for all $\pi \in \Irr \mathcal{R}_e (A^*)$. \label{K}
\end{prop}

\begin{proof} Let $\pi \in \Irr \mathcal{R}_e (A^*)$. There exists $s_\pi \in \mathcal{J}_e$ such that $\pi \in \Irr \mathcal{R}^{s_\pi}(A^*)$. For $s \in \mathcal{J}_e$ let $\rho_s \in \Irr \mathcal{R}^s(A^*)$ be a representative. Then $\pi$ is an unramified twist of $\rho_{s_\pi}$, say $\pi\cong \rho_{s_\pi}\otimes \chi$. Note that $\chi$ is trivial on $K$ and therefore $\dim(\pi^K)= \dim(\rho_{s_\pi}^K)\neq 0$. So for a general $\pi \in \Irr \mathcal{R}_e (A^*)$ we have  
$$\dim(\pi^K) \in \{\dim(\rho_s^K) : s\in \mathcal{J}_e\}.$$ 
The fact that $\chi$ is unramified implies that for any $K' \subset \GL_{md}(\mathcal{O}_F)$ open compact 
$$\dim(JL(\pi)^{K'})\in \{\dim(JL(\rho_{s_\pi})^{K'}) : s\in \mathcal{J}_e\},$$ 
as the Jacquet--Langlands transfer respects twists by a character. It is therefore sufficient to find a compact open subgroup $K_e \subset \GL_{md}(\mathcal{O}_F)$, such that $JL(\rho_{s_\pi})^{K_e} \neq 0$ for all $s\in \mathcal{J}_e$. Such a $K_e$ exists as $JL(\rho) = \bigcup_K JL(\rho)^K$ and $\mathcal{J}_e$ is finite.
\end{proof}

\begin{rem} Assume that $m=1$, so that $A^*=D^*$. As $D^*$ has no proper Levi subgroups defined over $F$ all irreducible smooth representations are supercuspidal, so we can apply Proposition \ref{K} for any idempotent $e \in C^\infty_c(D^*,\C)$. 
\end{rem}

\begin{lem}
Assume that $m=1$, so that $A^*=D^*$. For any $n \in \mathbb{N}$ there exists $t_n \in \mathbb{N}$, such that the following holds:
If $\pi=\chi \circ \operatorname{Nrd}:D^*\rightarrow F^*$ is a character of $D^*$ which is trivial on $K_n$ then $\pi':=\chi \circ \det: \GL_d(F)\rightarrow F^*$ is trivial on $K'_{t_n}:=\{ g \in \GL_d(\mathcal{O}_F): \det(g) \in 1+\varpi^{t_n}\mathcal{O}_F\}$.
\label{L}
\end{lem}
\begin{proof}
The units in the maximal order $O^*_D$ contain the units $O^*_L$ of an unramified extension $L/F$ of degree $d$, which implies that $\operatorname{Nrd}(O^*_D)=O^*_F$.
The index of $K_n$ in $O^*_D$ is finite, so the image $\operatorname{Nrd}(K_n)$ has finite index in $O^*_F$ and therefore contains a subgroup of the form $1+w^{t} O_F$ for some $t \in \mathbb{N}$. We can take $t_n$ to be any such $t$.
\end{proof}

\section{Eigenvarieties for definite unitary groups}
To set up notation, we will briefly recall the definition and existence theorem of eigenvarieties of idempotent type as constructed in Section 7 of \cite{bellaiche}. Let $E/\Q$ and $G$ be as in Section 2. Let $p$ be a prime that splits in $E$ and such that $G_p\cong\GL_3/\Q_p$. Fix such an isomorphism. Furthermore fix embeddings $\iota_p:\overline{\Q} \hookrightarrow \overline{\Q}_p$ and $\iota_\infty:\overline{\Q} \hookrightarrow \C$. 

Let $S_1$ be a finite set of places of $\Q$ containing the archimedean place, such that $G$ is unramified outside $S_1$. Fix a model $G/\Z_{S_1}$ of $G$. Let $\mu=\prod_l{\mu_l}$ be a product Haar measure on $G(\A_f)$ such that $\mu_l(\Z_l)=1$ for all $l\notin S_1$.

Let $S_0$ be a subset of the primes $l$ that split in $E$ and such that $G_l \cong \GL_3/\Q_l$. Assume furthermore that $p\notin S_0$ and $S_0 \cap S_1=\emptyset$. Define 
$$\mathcal{H}:= \mathcal{A}_p \otimes \mathcal{H}_{\mathrm{ur}},$$
where $\mathcal{H}_{\mathrm{ur}}=C_c(G(\widehat{\mathbb{Z}}_{S_0})\backslash G(\A_{S_0})/G(\widehat{\mathbb{Z}}_{S_0}), \Z)$ \footnote{ If $l$ is a prime which is inert in $E$, the extension $E_v/\Q_l$ is unramified and as $G_l\times \text{Spec}(E_{v}) \cong \GL_3$, the group $G$ is unramified at $l$. In the definition of $\mathcal{H}$ we could have included the (commutative) unramified Hecke algebras at these primes and all constructions could be made using this enlarged Hecke algebra.} and $\mathcal{A}_p$ is the commutative Atkin-Lehner subring of $C^\infty_c(I\backslash \GL_3(\Q_p)/I, \mathbb{Z}\left[\frac{1}{p}\right])$, where $I$ denotes the standard Iwahori subgroup of $\GL_3(\Q_p)$ (see \cite{bellaiche} Section 6.4.1 for the definition of $\mathcal{A}_p$).
For $u=\mathrm{diag}(p^2,p,1)\in \GL_3(\Q_p)$, let $[IuI] \in \mathcal{A}_p$ be the characteristic function on the double coset $IuI$. The element $ [IuI]\otimes 1_{\mathcal{H}_{\mathrm{ur}}} \in \mathcal{H}$ will be denoted by $u_0$.

Let $\mathcal{Z}_0\subset \hom_{\mathrm{ring}}(\mathcal{H},\overline{\Q}_p) \times \Z^3$ be the set of pairs $(\psi_{(\pi,\mathcal{R})},\underline{k})$ where $(\pi,\mathcal{R})$ runs over the set of $p$-refined automorphic representations, $\psi_{(\pi,\mathcal{R})}$ is the system of Hecke eigenvalues attached to $(\pi,\mathcal{R})$ and $\underline{k}$ is the weight of $\pi$ (see \cite{bellaiche} Section 7.2.2 for the definition of a $p$-refined automorphic representation and details regarding $\psi_{(\pi,\mathcal{R})}$). Let $\mathcal{Z}$ be a subset of $\mathcal{Z}_0$.

Weight space will be denoted by $\mathcal{W}$. It is the quasi-separated rigid analytic space over $\Q_p$ 
$$ \mathcal{W}:= \hom_{\mathrm{cont}}((\Z^*_p)^3,\mathbb{G}_m),$$
whose points over any affinoid $\Q_p$-algebra $A$ parameterize the continuous characters $(\Z^*_p)^3\rightarrow A^*$. 
We embed $\Z^3$ into $\mathcal{W}(\Q_p)$ by
$$ (k_1,k_2,k_3) \mapsto \left((t_1,t_2,t_3) \mapsto \prod_{i=1,...,3} t_i^{k_i}\right). $$
An eigenvariety for $\mathcal{Z}$ is defined as follows (for a definition of an accumulation and Zariski-dense subset of a rigid space see \cite{bellaiche} Section 3.3.1). 
\begin{Def} Let $L$ be a finite extension of $\Q_p$. 
An \textit{eigenvariety} $(X,\psi,\nu,Z)$ for $\mathcal{Z}$ over $L$ is a reduced $p$-adic rigid analytic space $X$ over $L$ together with 
\begin{itemize}
	\item a ring homomorphism $\psi: \mathcal{H} \rightarrow \mathcal{O}(X)^{rig} $,
	\item an analytic map $\omega: X\rightarrow \mathcal{W}$ defined over $L$ and
	\item an accumulation and Zariski-dense subset $Z \subset X(\overline{\Q}_p)$,
	\end{itemize}
such that the following conditions are satisfied:
\begin{enumerate}
	\item The map $\nu :=(\omega, \psi(u_0)^{-1}) : X \rightarrow \mathcal{W}\times \mathbb{G}_m$ is finite.
	\item For all open affinoid $V\subset \mathcal{W}\times \mathbb{G}_m$, the natural map 
	$$\psi \otimes \nu^*:\mathcal{H}\otimes_{\mathbb{Z}} \mathcal{O}(V) \rightarrow \mathcal{O}(\nu^{-1}(V)) $$
	is surjective.
	\item The natural evaluation map $X(\overline{\Q}_p)\rightarrow \hom_{\mathrm{ring}}(\mathcal{H},\overline{\Q}_p)$
      $$ x \mapsto \psi_x := (h \mapsto \psi(h)(x)) $$
	induces a bijection $Z \rightarrow \mathcal{Z}, z \mapsto (\psi_z, \omega(z))$.
\end{enumerate}
\end{Def}

Next we introduce the relevant sets $\mathcal{Z}$ that give rise to the eigenvarieties we want to compare via functoriality. 
Fix an idempotent 
$$e \in C_c^{\infty}(G(\A^{p,S_0}_f),\overline{\Q})\otimes 1_{\mathcal{H}_{\mathrm{ur}}} \subset C^{\infty}_c (G(\A^p_f),\overline{\Q})$$
and a finite extension $L$ of $\Q_p$, which contains the values of $i_p(e) := i_p\circ e$.
Let $\mathcal{Z}_e \subset \mathcal{Z}_0$ be the subset of $(\psi_{(\pi,\mathcal{R})},\underline{k})$, such that $e\cdot\pi_f^p\neq 0$.

Recall that a rigid space $X$ over $L$ is called \textit{nested} if it has an admissible covering by open affinoids $\{X_i,i\geq0\}$ such that $X_i \subset X_{i+1}$ and the natural $L$-linear map $\mathcal{O}(X_{i+1})\rightarrow \mathcal{O}(X_i)$ is compact. 
\begin{thm}
There exists a unique eigenvariety $(X/L,\psi,\nu,Z)$ for $\mathcal{Z}_e$. It is nested and equidimensional of dimension 3.
\end{thm} 
Proof. This is Theorem 7.3.1 in \cite{bellaiche}. The uniqueness assertion is proved in their Proposition 7.2.8. 

\section{\textit{p}-adic transfer of automorphic forms}
Let $G$ and $G'$ be unitary groups as in Section $2$. Fix models $G/\Z_{S_1}$ and $G'/\Z_{S_1}$ where we've chosen $S_1$, such that both groups are unramified outside $S_1$. Just as above let $S_0$ be a subset of the primes $l$ that split in $E$ and such that $G_l \cong \GL_3/\Q_l$, so in particular $S_{G}\cap S_0 = \emptyset$, where $S_G$ is as in Section 2. Note that our assumptions on $G'$ imply that for $l\in S_0$ we also have $G'_l \cong \GL_3/\Q_l$. Let $p \notin S_G$ be a prime that splits in $E$. Assume furthermore that $p\notin S_0$ and that $S_0\cap S_1=\emptyset$. 
Fix an idempotent
$$e \in C_c^{\infty}(G(\A^{p,S_0}_f),\overline{\Q})\otimes 1_{\mathcal{H}_{\mathrm{ur}}} \subset C^{\infty}_c (G(\A^p_f),\overline{\Q})$$ 
and assume 
\begin{equation}
e=\prod_{l\notin {S_0}\cup \{p\}} e_l \times \prod_{l\in S_0} e_{G(\Z_l)}
\label{eq:2}
\end{equation}
is a product of local idempotents $e_l \in C_c^{\infty}(G(\Q_l),\overline{\Q})$. 

As in Section 2, we let $S$ be the set of places where the local groups are not isomorphic. For $l\in S$ choose a compact open subgroup $K_l\subset G'(\Q_l)$ such that 
$(\pi'_l)^{K_l}\neq0$ 
for all local components $\pi'_l$ of automorphic representations $\pi'$ of $G'$ that are the transfer of an automorphic representation $\pi$ of $G$ satisfying $e\cdot \pi_f^p \neq0$. The existence of such compact open subgroups is guaranteed by Proposition \ref{K} and Lemma \ref{L}. Define the idempotent
$$e':= \prod_{l \in S} e_{K_l} \times \prod_{l \notin S\cup S_0\cup \{p\}} e_l \times \prod_{l\in S_0} e_{G'(\Z_l)}.$$
Note that for any $p$-refined automorphic representation $(\pi,\mathcal{R})$ of $G$
$$ \psi_{(\pi,\mathcal{R})}= \psi_{(\pi',\mathcal{R})},$$ 
where $\pi'$ is the transfer of $\pi$ from Theorem \ref{transfer}. We have established an inclusion
\begin{equation}
\mathcal{Z}_e \subset \mathcal{Z}_{e'}.
\label{eq:1}
\end{equation}

\begin{lem}
Let $\mathcal{Z}_1 \subset \mathcal{Z}_2$ be two subsets of $\mathcal{Z}_0 \subset \hom_{\mathrm{ring}}(\mathcal{H},\overline{\Q}_p)\times \mathbb{Z}^3$. Assume there exist eigenvarieties $(X_1/L,\psi_1,\nu_1,Z_1)$ for $\mathcal{Z}_1$ and $ (X_2/L,\psi_2,\nu_2,Z_2)$ for $\mathcal{Z}_2$. Then there exists a unique closed immersion 
$$\zeta : X_1 \hookrightarrow X_2$$ 
over $L$ such that $\nu_1=\zeta \circ \nu_2$ and for all $h \in \mathcal{H}, \ \psi_1(h)= \zeta^* (\psi_2(h))$.
\end{lem}
\begin{proof}
By Lemma 7.2.7 (b) of \cite{bellaiche} and property 3 of an eigenvariety there exists a unique injection $\zeta: Z_1 \hookrightarrow Z_2$, such that for all $z \in Z_1, \psi_{1_z} = \psi_{2_{\zeta(z)}} $ and $\omega_1(z)=\omega_2(\zeta(z))$. Our claim is that we can extend $\zeta$ to a closed immersion $\zeta:X_1 \hookrightarrow X_2$.
Using the fact that the $X_i$ are reduced and Lemma 7.2.7 (b) of \cite{bellaiche} again we see that such an extension if it exists is unique. 
The proof of the existence is contained verbatim in the proof of the uniqueness of eigenvarieties, that is in the proof of Proposition 7.2.8 of \cite{bellaiche}. More precisely they construct a closed immersion $X_1 \hookrightarrow X_2$. The same construction applies in our situation and proves the Lemma.
\end{proof}

\begin{thm}
Let $G$, $G'$, $S_0$ ,$e$ and $e'$ be as above. Let $L/\Q_p$ be a finite extension that contains the values of $i_p(e)$ and of $i_p(e')$. \\
Let $(X/L,\psi,\nu,Z_e)$ (resp. $(X'/L,\psi',\nu',Z{_e'})$) be the eigenvariety for $\mathcal{Z}_e$ (resp. $\mathcal{Z}_{e'}$).
Then there exists a unique closed immersion of eigenvarieties 
$$\zeta: X \hookrightarrow X' $$
compatible with the inclusion \ref{eq:1}, i.e., with the classical transfer of automorphic representations.
The image of $\zeta$ is a union of irreducible components. 
\label{main}
\end{thm}
\begin{proof}
Everything except the last assertion follows from the previous lemma. The last assertion follows from \cite{conrad} Corollary 2.2.7, which states that for a rigid space $X$, which is equidimensional of dimension $m$, the analytic subsets in $X$, which are equidimensional of dimension $m$, are exactly the union of irreducible components.  
\end{proof}

\end{document}